\documentclass[11pt]{amsart}
\usepackage{graphicx,amsfonts,amssymb,amsmath,amsthm,url,
  verbatim,amscd}
  \usepackage{color}
\usepackage[usenames,dvipsnames]{xcolor}
\usepackage[normalem]{ulem}
\usepackage{pdfsync}
\usepackage{fullpage}

\usepackage[hyperfootnotes=false, colorlinks, citecolor=RoyalBlue, urlcolor=blue, linkcolor=blue ]{hyperref}

\newtheorem{lemma}{Lemma}[section]
\newtheorem{theorem}{Theorem}

\newtheorem{corollary}[lemma]{Corollary}
\newtheorem{proposition}[lemma]{Proposition}

\theoremstyle{remark}
\newtheorem{remark}[lemma]{Remark}

\allowdisplaybreaks

\renewcommand{\H}{\mathbb H}
\newcommand{\A}{{\mathbb A}}
\newcommand{\Q}{{\mathbb Q}}
\newcommand{\Z}{{\mathbb Z}}

\newcommand{\R}{{\mathbb R}}
\newcommand{\C}{{\mathbb C}}
\newcommand{\bs}{\backslash}

\newcommand{\p}{\mathfrak p}

\newcommand{\GL}{{\rm GL}}

\newcommand{\SL}{{\rm SL}}

\newcommand{\sgn}{{\rm sgn}}

\newcommand{\trace}{{\rm tr}}

\def\PSL{\operatorname{PSL}}

\def\SL{\operatorname{SL}}

\def\GL{\operatorname{GL}}

\renewcommand{\Re}{\mathrm{Re}}
\renewcommand{\Im}{\mathrm{Im}}

\begin{document}

\bibliographystyle{alpha}

\title{On sup-norms of cusp forms of powerful level}

\author{Abhishek Saha}
\address{Department of Mathematics\\
  University of Bristol\\
  Bristol BS81TW\\
  UK} \email{abhishek.saha@bris.ac.uk}

\begin{abstract} Let $f$ be an $L^2$-normalized Hecke--Maass cuspidal newform of level $N$
and Laplace eigenvalue $\lambda$. It is shown that $\|f\|_\infty \ll_{\lambda, \epsilon} N^{-1/12 + \epsilon}$ for any $\epsilon>0$. The exponent is further improved in the case when $N$ is not divisible by ``small squares". Our work extends and generalizes previously known results in the special case of $N$ squarefree.
\end{abstract}
\maketitle

\def\eachnotany{~{each}~}

\section{Introduction}The problem of bounding the sup-norms of $L^2$-normalized cuspidal automorphic forms has been much studied recently, beginning with the work of Iwaniec and Sarnak~\cite{iwan-sar-85}, who proved the first non-trivial bound in the eigenvalue-aspect for Hecke--Maass cusp forms. Since then, this question has been considered in the eigenvalue/weight~\cite{koyama,  vanderkam,  donnelly, rud05, xia, das-sengupta, brumley-templier,  blomer-pohl, hol-ricotta-royer, blomer-maga}, volume/level~\cite{abbes-ullmo, jorgenson-kramer,  lau, templier-sup, harcos-templier-1, harcos-templier-2, templier-large, kiral} and hybrid~\cite{blomer-holowinsky, blomer-michel,  templier-sup-2, blomer-harcos-milicevic} aspects for various types of automorphic forms. One reason why this problem is interesting is its connections with various other topics, such as the theory of quantum
chaos, the subconvexity of $L$-functions, the combinatorics of Hecke-algebras, and diophantine analysis.

Our interest in this paper is in the level aspect. We consider the sup-norm question for eigenfunctions on the arithmetic hyperbolic surface $\Gamma_0(N)\bs \H$ equipped with the measure $\frac{dx dy}{y^2}$. It is natural to restrict to the case of newforms. Thus, we are interested in bounding the sup-norms of $L^2$-normalized Hecke--Maass newforms $f$ of level $N$ (and trivial character) in the $N$-aspect. The following upper bounds for $\|f\|_\infty$ in the $N$-aspect were known prior to this work:
\begin{itemize}
\item The ``trivial bound" $\|f\|_\infty \ll_{\lambda, \epsilon} N^{ \epsilon}$.

\item $\|f\|_\infty \ll_{\lambda, \epsilon} N^{-\frac{25}{914} + \epsilon}$ for squarefree $N$, due to Blomer and Holowinsky~\cite{blomer-holowinsky}, published in 2010.

\item  $\|f\|_\infty \ll_{\lambda, \epsilon} N^{-\frac{1}{22} + \epsilon}$ for squarefree $N$, due to Templier~\cite{templier-sup}, published in 2010.

  \item  $\|f\|_\infty \ll_{\lambda, \epsilon} N^{-\frac{1}{20} + \epsilon}$ for squarefree $N$, due to Helfgott--Ricotta (unpublished).

 \item $\|f\|_\infty \ll_{\lambda, \epsilon} N^{-\frac{1}{12} + \epsilon}$ for squarefree $N$, due to Harcos and Templier~\cite{harcos-templier-1}, published in 2012.

  \item $\|f\|_\infty \ll_{\lambda, \epsilon} N^{-\frac{1}{6} + \epsilon}$ for squarefree $N$ due to Harcos and Templier~\cite{harcos-templier-2}, published in 2013.\footnote{Templier, in separate work~\cite{templier-sup-2}, has successfully combined this bound with the bound of Iwaniec--Sarnak in the eigenvalue-aspect, to obtain a state-of-the-art hybrid estimate.}

\end{itemize}

As the above makes clear, there has been fairly rapid progress in the squarefree case, yet no improvement has been obtained beyond the trivial bound when $N$ is not squarefree. Indeed, all the above papers rely crucially on using Atkin-Lehner operators to move any point of $\H$ to a point of  imaginary part $\ge \frac1N$  (which is essentially equivalent to using a suitable Atkin--Lehner operator to move any cusp to infinity). This only works if $N$ is squarefree.

In this paper, we introduce some new ideas and technical improvements which allows us to obtain a non-trivial result without any square-free assumptions.\footnote{A look at the wider sup-norm literature suggests that this is the first time that the squarefree  barrier  has  been non-trivially broken for any kind of automorphic form on a domain that contains cusps.}

\begin{theorem}\label{t:main}Let $f$ be an $L^2$-normalized Hecke-Maass cuspidal newform for the group $\Gamma_0(N)$ with Laplace eigenvalue $\lambda \le T$.
\begin{enumerate}
\item For any $\epsilon>0$ we have the bound $$\|f\|_\infty \ll_{T,\epsilon} N^{-1/12 + \epsilon}.$$

\item \label{case2} Suppose that there is no integer $M$ in the range $1<M<N^{1/6}$ such that $M^2$ divides $N$. Then we can improve the above bound to  $$\|f\|_\infty \ll_{T,\epsilon} N^{ \epsilon} \max(N^{-1/6}, N^{-1/4}N_0^{1/4})$$ where $N_0$ is the largest integer such that $N_0^2$ divides $N$. In particular, in this case we always have $$\|f\|_\infty \ll_{T,\epsilon} N^{-1/8+ \epsilon}.$$
\end{enumerate}

\end{theorem}

\begin{remark}Assertion~\eqref{case2} of the Theorem can be regarded as dealing with the case when $N$ is not divisible by ``small squares". This includes, for instance, the squarefree case (in which case we recover the bound $\|f\|_\infty \ll_{T,\epsilon} N^{-1/6+ \epsilon}$ due to Harcos--Templier), the case $N=p^2N_2$ where $N_2$ is squarefree and $p$ is a prime such that $p \ge N_2^{1/4}$, and the  case $N = p^n$ where $p$ is a prime and $1\le n \le 6$.
\end{remark}

\begin{remark}
All the results of this paper (and in particular the main result above) remain valid in the case of holomorphic newforms of fixed weight and varying level $N$.
\end{remark}

\begin{remark}In this paper we have restricted for simplicity to the case of trivial central character. We have also made no effort to obtain a hybrid bound, i.e., we haven't attempted to quantify the dependence of our constants on the Laplace eigenvalue. However, we expect that the methods of this paper, with some modifications, will be able to deal with these cases. Further, we hope that this paper will shed some light on how to remove the squarefree restriction from sup-norm bounds for more general automorphic forms. We will come back to some of these questions in future work.
\end{remark}

\begin{remark} \textbf{(Added in proof)} Recently, we have succeeded in significantly improving the results of this paper, as well as obtained a hybrid bound. This is done in our forthcoming work \cite{saha-sup-level-hybrid}, which uses a fairly different (and in our view,  superior) adelic methodology compared to this paper.
\end{remark}

\medskip

Let us briefly explain the new ingredients in this paper compared to the paper by Harcos and Templier~\cite{harcos-templier-2} (whose general strategy we broadly follow). Our key new idea is to look at the behavior of cusp forms around \emph{cusps of width 1}. Recall that if $N$ is squarefree, then the surface $\Gamma_0(N)\bs \H$ has exactly one cusp of width 1, namely the cusp at infinity. However, if $N$ is not squarefree, then there is always more than one cusp of width 1. The cusps of width 1 have several nice properties. First, any cusp can be conjugated to a cusp of width 1 by use of a suitable Atkin-Lehner operator. Secondly, this leads to a ``gap principle", whereby any point of $\H$ can be moved by an Atkin-Lehner operator to another point which has high imaginary part and good diophantine properties when re-written in the coordinates corresponding to a suitable cusp of width 1. Thirdly, if $\sigma \in \SL_2(\Z)$ is a matrix that takes the cusp at infinity to a cusp of width 1, then for any Hecke-Maass cuspidal newform $f$ for $\Gamma_0(N)$, the function $f|\sigma$ is a Maass cusp form on the slightly smaller group $\Gamma_0(N;M):= \Gamma_0(N) \cap \Gamma_1(M)$ (where $M^2$ is a suitable divisor of $N$) and moreover $f|\sigma$ is an eigenfunction of the Hecke operators at all primes congruent to 1 mod $M$.

We exploit the above facts to reduce the sup norm question from $f$ to some suitable $f|\sigma$. However, several technical difficulties arise. First, the counting problem that lies at the heart of the amplification method becomes much more involved, especially for the \emph{parabolic} matrices. Secondly, the bound via applying the Cauchy--Schwarz inequality on the Fourier expansion requires us now to undertake a deep study of the Fourier coefficients at the cusp $\sigma$. Thirdly, because the surface $\Gamma_0(N;M)\bs \H$ has higher volume than $\Gamma_0(N)\bs \H$ and because we can now amplify \emph{only} over primes that are 1 mod $M$, we lose some sharpness in our bounds, and it is important to offset this in some way\footnote{This is achieved by a twofold process. First, our gap principle contains a factor of $M^2$, which makes the bounds obtained via the Fourier expansion extremely strong when $M$ is relatively large. Secondly, our counting arguments are refined to mostly account for the presence of $M$.} so that this loss is not too prominent. These technical difficulties are however, all successfully overcome, and in the end we get the theorem quoted above.
\bigskip

We end this introduction with a few speculative remarks regarding the true order of magnitude for $\|f\|_\infty$. The trivial \emph{lower bound} for $\|f\|_\infty$  in the $N$-aspect is $\|f\|_\infty \gg_{T,\epsilon} N^{-1/2 - \epsilon}$ and this bound is also valid for $L^2$-normalized Hecke--Maass newforms with non-trivial character.  However, if the conductor of the character is large relative to $N$, local effects (coming from the behavior of local Whittaker newforms for ramified principal series representations) lead to stronger lower bounds. For example, if $f$ is an $L^2$-normalized Hecke--Maass newform of level $N$ with $N$ a perfect square, and the conductor of the character attached to $f$ is also equal to $N$, then Templier~\cite{templier-large} showed that $\|f\|_\infty \gg_{T,\epsilon} N^{-1/4 - \epsilon}$. In forthcoming work by the author \cite{saha-sup-level-hybrid, sahasupwhittaker}, the results of this paper, as well as Templier's example, will be generalized to a wide variety of cases with non-trivial character. Moreover, we will precisely measure the local effects coming from the ramified Whittaker newforms, and thus will be able to make a conjecture about the true size of $\|f\|_\infty$. In the case of trivial central character as in this paper, or more generally if the exponent of each prime dividing the conductor of the character is at most half the exponent of the prime dividing the square-ful part of $N$, we will (optimistically) conjecture that $N^{-1/2 - \epsilon} \ll_{T,\epsilon} \|f\|_\infty \ll_{T,\epsilon} N^{-1/2 + \epsilon}.$

\subsection*{Acknowledgements} I would like to thank Martin Dickson, Gergely Harcos, Emmanuel Kowalski, Paul Nelson and Guillaume Ricotta for helpful comments.
\subsection*{Some basic notations and definitions}
\begin{itemize}

\item The symbols $\Z$, $\Z_{\ge0}$, $\Q$, $\R$, $\C$, $S^1$, $\Z_p$ and $\Q_p$ have the usual meanings. $\A$ denotes the ring of adeles of $\Q$.

\item For any two complex numbers $\alpha, z$, we let $K_\alpha(z)$ denote the modified Bessel function of the second kind. We write $e(z) := e^{2 \pi
  i z}$. For each positive integer $n$, we let $\phi(n)$ denote the Euler
phi function
$\phi(n) = \# (\mathbb{Z}/n)^\times = \# \{a \in \mathbb{Z} : 1 \leq a
\leq n, (a,n) = 1\}$.

\item Given two integers $a$ and $b$, we use $a|b$ to denote that $a$ divides $b$, and we use $a|b^\infty$ to denote that $a|b^n$ for some positive integer $n$. We use $(a,b)$ to denote the greatest common divisor of $a$ and $b$, which by our convention is always positive. We use $(a,b^\infty)$ to denote the limit $\lim_{n \rightarrow \infty}(a, b^n)$, which always exists. We write $a^n||b$ to mean that  $a^n | b$ and $a^{n+1}$ does not divide $b$. For any real number $\alpha$, we let $\lfloor \alpha \rfloor$ denote the greatest integer less than or equal to $\alpha$ and we let $\lceil \alpha \rceil$ denote the smallest integer greater than or equal to $\alpha$.

\item For any commutative ring $R$ and positive integer $n$,
$M_n(R)$ denotes the ring of $n$ by $n$ matrices with entries in $R$ and $GL_n(R)$ denotes the group of
invertible matrices in $M_n(R)$. We use $R^{\times}$ to denote $GL_1(R)$.

\item The groups $\SL_2$, $\PSL_2$ and $\Gamma_0(N)$ have their usual meanings. We let $\GL_2^+(\R)$ denote the subgroup of $\GL_2(\R)$ consisting of matrices with positive determinant.

\item  We let $\H = \{x + iy: x \in \R, \ y\in \R, \ y>0\}$ denote the upper half plane. For any $\gamma=\begin{pmatrix}a&b\\c&d\end{pmatrix}$ in $\GL_2^+(\R)$, and any $z \in \H$, we define $\gamma(z)$ or $\gamma z$ to equal $\frac{az+b}{cz+d}$. This action of $\GL_2^+(\R)$ on $\H$  extends naturally to the boundary of $\H$. For any $g \in \GL_2^+(\R)$ and any function $f$ on $\H$, we let $f|\gamma$ denote the function on $\H$ defined by $f|\gamma(z)=f(\gamma z)$.

 \item For any congruence subgroup $\Gamma$ of $\SL_2(\Z)$, and any bounded function $f:\H \rightarrow \C$ satisfying $f(\gamma z)=f(z)$ for all $\gamma \in \Gamma$, we define $\langle f, f \rangle_\Gamma = \int_{\Gamma \bs \H} |f(z)|^2 \frac{dx dy}{y^2},$ and $\|f\|_\infty = \sup_{z \in \H} f(z)$.   We say that such a $f$ is a Maass cusp form for/on $\Gamma$ if $f$ is an eigenfunction of the hyperbolic Laplacian $\Delta := y^{-2}
(\partial_x^2 + \partial_y^2)$ on $\mathbb{H}$ and decays
rapidly at the cusps of $\Gamma$.  The Laplace eigenvalue of such a  Maass cusp form $f$ is the real number $\lambda$ satisfying $(\Delta + \lambda) f =
0$. We can write $\lambda = 1/4 + r^2$ where
 $r \in \mathbb{R} \cup
i[0,1/2]$;
this follows from the nonnegativity of $\Delta$. We say that $f$ is $L^2$-normalized if $\langle f, f \rangle_\Gamma = 1$.

 \item We say that $f$ is a cuspidal Hecke--Maass newform for $\Gamma_0(N)$ (also referred to as a cuspidal Hecke--Maass newform of level $N$ and trivial character) if it is a Maass cusp form for $\Gamma_0(N)$ and is a newform in the sense of Atkin--Lehner (i.e., it is orthogonal to all oldforms, and is an eigenfunction of all  the Hecke and Atkin-Lehner operators). A cuspidal Hecke--Maass newform $f$ is always either even or odd, i.e., there exists $\epsilon_f \in \{\pm 1 \}$ such that $f(-\overline{z}) = \epsilon_f f(z).$

\item We use the notation
$A \ll_{x,y,z} B$
to signify that there exists
a positive constant $C$, depending at most upon $x,y,z$,
so that
$|A| \leq C |B|$.

\item The symbol $\epsilon$ will denote a small positive quantity, whose value may change from line to line, and the value of the constant implicit in $\ll_{\epsilon, \ldots}$ may also change from line to line.   An assertion such as ``Let $1 \le L \le N^{O(1)}.$ Then $f(\epsilon, L, N, \ldots) \ll_{\epsilon, \ldots} N^{O(\epsilon)} g(L,N, \ldots)$" means ``For every $C>0, \ 1 \le L \le N^C,$ there is a constant $D>0$ depending only on $C$ such that $f(\epsilon, L, N, \ldots) \ll_{ C, \epsilon, \ldots} N^{D\epsilon}  g(L,N, \ldots).$"
\end{itemize}

\section{Cusps of width 1 and Atkin-Lehner operators}
Let $N=\prod_p p^{n_p}$ be a positive integer. Let $\mathbb{P}^1(\mathbb{Q})$ denote the set of all boundary points  of the upper-half plane $\H$ that are stabilized by a non-trivial element of $\PSL_2(\Z)$; precisely, $\mathbb{P}^1(\mathbb{Q})$  is the union of $\infty$ and the rational points on the real line.  The set  $\mathcal{C}(\Gamma_0(N))= \Gamma_0(N) \bs \mathbb{P}^1(\mathbb{Q})$ is the set of cusps of $\Gamma_0(N)$. For any ring $R$, let $N(R)=\{\left(
  \begin{smallmatrix}
    1&n\\
    &1
  \end{smallmatrix}
\right): n
 \in R\}$ and $Z(R)=\{\left(
  \begin{smallmatrix}
    z&\\
    &z
  \end{smallmatrix}
\right): z\in R^\times\}$. Via the correspondence $z \leftrightarrow \gamma(\infty)$, the set $\mathcal{C}(\Gamma_0(N))$  can be identified with the double coset space $\Gamma_0(N) \bs \SL_2(\Z) / N(\Z)$. Given any $\tau \in \SL_2(\Z)$, we can therefore speak of (some property of) the cusp (corresponding to) $\tau$.

Let $\tau \in \SL_2(\Z)$. The cusp $\tau (\infty)$  contains a representative of the form $\frac{a}c$, where $a,c \in \Z, \  c|N$, $c>0$, $\gcd(a,c)=1$. The integer $c$ is uniquely determined. We will denote $C(\tau)=c$ and refer to $C(\tau)$ as the denominator of the cusp corresponding to $\tau$. It can be easily checked that if $\tau = \begin{pmatrix}a&b\\c&d\end{pmatrix}$ then $C(\tau) = \gcd(c, N)$. If  $\Gamma_0(N)\tau_1 N(\Z) = \Gamma_0(N)\tau_2 N(\Z)$, then $C(\tau_1)=C(\tau_2)$.

We let $W(\tau) $ denote the width of the cusp corresponding to $\tau$; precisely, $W(\tau) = [ N(\Z) : N(\Z) \cap
\tau^{-1} \Gamma_0(N) \tau]$. Since the group $N(\Z) \cap
\tau^{-1} \Gamma_0(N) \tau$ always contains $\{ \left(
  \begin{smallmatrix}
    1&Nn\\
    &1
  \end{smallmatrix}
\right): n \in \mathbb{Z}\}$ which has index $N$ in $N(\Z)$, it follows  that $W(\tau)$ is a divisor of $N$.

For the convenience of the reader, we note down a few standard facts, proofs of which can be found for example in~\cite[Sec. 3.4.1]{NPS}.
\begin{itemize}
\item For each $c|N$, the number of cusps with denominator $c$ equals $\phi((c, N/c))$. Thus, the total number of cusps equals $\sum_{c|N} \phi((c, N/c))$. Moreover, there exists only one cusp of denominator $N$, namely the cusp $\infty$ ($=1/N$).

\item For each $\tau \in \SL_2(\Z)$ we have $W(\tau) = N/ (C(\tau)^2, N).$ In particular, $W(\tau)=1$ if and only if $N|C(\tau)^2$.

 \item If $N$ is squarefree, then there is exactly one cusp of width 1, namely the cusp $\infty$. However, if $N$ is not squarefree, then there is always more than one cusp of width 1.

 \end{itemize}

\begin{remark}From the above facts, it is clear that an element $\sigma \in \SL_2(\Z)$ satisfies $W(\sigma)=1$ if and only if $C(\sigma) = N/M$ for some positive integer $M$ such that $M^2|N$. \end{remark}

For each prime $p$, let $$K_0(p^{n_p}) =  \GL_2(\Z_p) \cap
  \begin{pmatrix}
    \Z_p & \Z_p  \\
  p^{n_p}\Z_p  & \Z_p \end{pmatrix}.
$$ For any divisor $M$ of $N$, we define the congruence subgroup $\Gamma_0(N;M)$ as follows:
  $$\Gamma_0(N;M) = \{\gamma = \begin{pmatrix}a&b\\c&d\end{pmatrix} \in \SL_2(\Z): a\equiv d \equiv 1 \bmod{M}, \ c \equiv 0 \bmod{N}\}.$$ Note that $\Gamma_0(N;1) = \Gamma_0(N)$ and $\Gamma_0(N;N)= \Gamma_1(N)$. We have the following proposition.

\begin{proposition}\label{gammanmcompatibility}Suppose that $M$ is a positive integer such that $M^2$ divides $N$. Let $\sigma \in \SL_2(\Z)$ be such that $C(\sigma) = N/M$. Then $\sigma \Gamma_0(N;M) \sigma^{-1} \subseteq \Gamma_0(N;M).$

\end{proposition}
\begin{proof}Recall that $C(\sigma) = N/M$ iff the lower left entry of $\sigma$ is a multiple of $N/M$. Now the result follows from the equation \begin{equation}\label{eq:keygammanm}\begin{split}&\begin{pmatrix}a&b\\Nc/M &d \end{pmatrix}^{-1}\begin{pmatrix}1+Mp & q\\Nr &1+Ms \end{pmatrix}\begin{pmatrix}a&b\\Nc/M &d \end{pmatrix}\\ &= \begin{pmatrix}1 + Madp + N(bdr-bcs) -(N/M)acq   & a^2q +bM (as - ap -br(N/M)) \\ N (dcp + d^2r - cds - c^2q(N/M^2)) & 1  + Mads + N(-bdr-bcp) +(N/M)acq  \end{pmatrix}.\end{split}\end{equation}
\end{proof}

Let $\mathcal{P}_N$ denote the set of primes dividing $N$. For each subset $S \subseteq \mathcal{P}_N$, we define $N_S = \prod_{p\in S}p^{n_p}$ where we understand $N_\emptyset=1$. We let $\mathcal{W}(S)$ denote the set $$\mathcal{W}(S) = \{W \in M_2(\Z): W \equiv \begin{pmatrix}0&*\\0&0\end{pmatrix} \pmod{N_S}, \ W \equiv \begin{pmatrix}*&*\\0&*\end{pmatrix} \pmod{N}, \ \det(W)=N_S \}.$$

The elements of $\mathcal{W}(S)$ (considered as operators on $\H$) are called the Atkin-Lehner operators. It is well-known~\cite{MR0268123} that all elements $W$ in $\mathcal{W}(S)$ satisfy $W^2 \in Z(\Q)\Gamma_0(N)$. The main other property of an Atkin-Lehner operator $W$ we need is that $$W \in \begin{cases} K_0(p^{n_p})\begin{pmatrix}0&-1\\p^{n_p}&0
\end{pmatrix}, & \text{ if } p \in S,  \\ K_0(p^{n_p}), &  \text{ if } p \notin S, \end{cases}$$ which follows directly from the definitions.

\begin{proposition}\label{keyprop1}Let $\tau \in \SL_2(\Z)$. Then there exists a subset $S$ of $\mathcal{P}_N$, an Atkin-Lehner operator $W \in \mathcal{W}(S)$,  positive integers $M_1, M$ such that $M^2|N$, $M_1= (M, N_S)$, $M_1^2|N_S$,  and an element $n \in N(\Q)$, such that the matrix $\sigma$ defined by $\sigma = W \tau n\begin{pmatrix} 1/M_1 &0\\0&M_1/N_S\end{pmatrix}$ has the following properties:
\begin{enumerate}
\item $\sigma \in \SL_2(\Z)$,
\item $C(\sigma)=N/M$.
\end{enumerate}

\end{proposition}

The proof will be essentially local in nature.
For any $\tau = \left(
  \begin{smallmatrix}
    a & b  \\
  c  & d\end{smallmatrix}
\right) \in \GL_2(\Z_p)$, we define $c_p(\tau) = \min(v_p(c), n_p)$ and
$w_p(\tau) = \max(n_p - 2c_p(\tau), 0)$. Note that the integers $c_p(\tau)$ and $w_p(\tau)$ both range between 0 and $n_p$.
\begin{lemma}\label{doublecosetrep} The integers $c_p(\tau)$ and $w_p(\tau)$ depend only on the double coset $K_0(p^{n_p})\tau N(\Z_p).$ Moreover, the double coset $K_0(p^{n_p})\tau N(\Z_p)$ contains the matrix $\begin{pmatrix} 1&0\\p^{c_p(\tau)}&v\end{pmatrix}$ for some $v \in \Z_p^\times$.
\end{lemma}
\begin{proof} The first assertion is immediate by looking at the matrix products in $K_0(p^{n_p})\tau N(\Z_p)$ modulo $p^{n_p}.$  For the second assertion, we consider three cases for $\tau = \left(
  \begin{smallmatrix}
    a & b  \\
  c  & d\end{smallmatrix}\right).$ The first case is when $v_p(c)=0$. In this case, $c$ is a unit and the result follows from the equation   $$\begin{pmatrix}1&(1-a)/c\\0&1/c\end{pmatrix}\begin{pmatrix}
    a & b  \\
 c  & d\end{pmatrix}\begin{pmatrix}1&(ad-bc-d)/c\\0&1\end{pmatrix} = \begin{pmatrix}1&0\\1&(ad-bc)/c\end{pmatrix}.$$

  The second case is $0<k=v_p(c)<n_p$. In this case, $a,d$ and $c_1=c/p^k$ are all units. The result follows from the equation   $$\begin{pmatrix}1/a&0\\0&1/c_1\end{pmatrix}\begin{pmatrix}
    a & b  \\
  p^kc_1  & d\end{pmatrix}\begin{pmatrix}1&-b/a\\0&1\end{pmatrix} = \begin{pmatrix}1&0\\p^k&d/c_1 +p^kb/a\end{pmatrix}.$$

  The third case is $0<n_p \le k=v_p(c)$. In this case, $a$ is a unit. The result follows from the equation   $$\begin{pmatrix}1/a&0\\(p^{n_p}-c)/a&1\end{pmatrix}\begin{pmatrix}
    a & b  \\
  c  & d\end{pmatrix}\begin{pmatrix}1&(1-c/p^{n_p})/a\\0&1\end{pmatrix} = \begin{pmatrix}1&0\\p^{n_p}&(ad-bc)/a\end{pmatrix}.$$

\end{proof}

\begin{lemma}Let $\tau \in \SL_2(\Z)$. Then $W(\tau) = \prod_{p|N} p^{w_p(\tau)}$ and $C(\tau)= \prod_{p|N}p^{c_p(\tau)}$.
\end{lemma}
\begin{proof} The equation $C(\tau)= \prod_{p|N}p^{c_p(\tau)}$ follows immediately from the relevant definitions. The relation $W(\tau) = \prod_{p|N} p^{w_p(\tau)}$ follows from the formulas $w_p(\tau) = \min(n_p - 2c_p(\tau), 0)$ and $W(\tau) = N/ (C(\tau)^2, N).$
\end{proof}

\begin{lemma}\label{keylemma1}Let $p|N$ and suppose that $\tau \in \GL_2(\Z_p)$ satisfies $w_p(\tau)>0.$ Then there exists $n \in N(\Q_p)$ such that for all $y_1, y_2 \in \Z_p^\times$, $\gamma\in K_0(p^{n_p})$, the element $\sigma \in \GL_2(\Q_p)$ defined via $$\sigma= \gamma\begin{pmatrix}0&-1\\p^{n_p}&0\end{pmatrix} \tau n \begin{pmatrix} y_1p^{-c_p(\tau)} &0\\0&y_2p^{c_p(\tau)-n_p}\end{pmatrix}$$ has the following properties:
\begin{enumerate}
\item $\sigma \in \GL_2(\Z_p)$,
\item $c_p(\sigma)=n-c_p(\tau)$,
\end{enumerate}
\end{lemma}
\begin{proof}Note that if $\sigma$ satisfies the required properties, then so does all elements in $K_0(p^{n_p})\sigma$. Hence we may assume that $\gamma=1$. Moreover, since $\begin{pmatrix}0&-1\\p^{n_p}&0\end{pmatrix}$ normalizes $K_0(p^{n_p})$, it follows that if the Proposition is true for some $\tau$, it is true for all $\tau \in K_0(p^{n_p})\tau N(\Z_p).$ Hence, using Lemma~\ref{doublecosetrep} we can assume without loss of generality that $\tau = \begin{pmatrix} 1&0\\p^{c_p(\tau)}&v\end{pmatrix}$ for some $v \in \Z_p^\times$. Define $n = \begin{pmatrix}
    1&-v p^{-c_p(\tau)}\\
    &1
  \end{pmatrix}
$. The condition $w_p(\tau)>0$ is equivalent to $2{c_p(\tau)}<n_p$. Then $$\sigma= \begin{pmatrix}0&-1\\p^{n_p}&0\end{pmatrix} \tau n\begin{pmatrix} y_1p^{-c_p(\tau)} &0\\0&y_2p^{c_p(\tau)-n_p}\end{pmatrix} = \begin{pmatrix}-y_1&0\\p^{n_p - c_p(\tau)}y_1& -vy_2 \end{pmatrix}.$$ So by inspection, we see that $\sigma \in \GL_2(\Z_p)$ and $c_p(\sigma) = n_p - c_p(\tau)$.

\end{proof}
\begin{lemma}\label{keylemma2}Let $p|N$ and suppose that $\tau \in \GL_2(\Z_p)$ satisfies $w_p(\tau)=0.$ Then for all $y_1, y_2 \in \Z_p^\times$, $\gamma\in K_0(p^{n_p})$, $n \in N(\Z_p)$, the element $\sigma \in \GL_2(\Z_p)$ defined via $$\sigma= \gamma \tau n \begin{pmatrix} y_1 &0\\0&y_2\end{pmatrix}$$ satisfies
$c_p(\sigma)=c_p(\tau)$,
\end{lemma}
\begin{proof}This is immediate from the definitions.
\end{proof}

We now prove Proposition~\ref{keyprop1}.
\begin{proof}[Proof of Prop.~\ref{keyprop1}] Let $S$ be the set of primes $p$ for which $w_p(\tau)>0$ (i.e., $c_p(\tau) < n_p/2$). Define $M_1 = \prod_{p\in S}  p^{c_p(\tau)}$; note that $M_1^2 | N_S$. For each $p|N$, we put $m_p = c_p(\tau)$ if $p \in S$ and $m_p = n_p-c_p(\tau)$ if $p \notin S$. Define $$M = M_1 \prod_{p|N, p \notin S} p^{n_p - c_p(\tau)} =
\prod_{p|N} p^{m_p}.$$ Note that $M^2 |N$. Pick $W$ to be any element of $\mathcal{W}(S)$. Note that $W$ considered as an element of $\GL_2(\Q_p)$ lies in  $K_0(p^{n_p})\begin{pmatrix}0&-1\\p^{n_p}&0
\end{pmatrix}$ for each $p$ in $S$ and lies in $K_0(p^{n_p})$ for each prime outside $S$. For each $p \in S$, Lemma~\ref{keylemma1} provides an element $x_p \in N(\Q_p)$ such that $$\sigma_p= W \tau x_p \begin{pmatrix}1/M_1 &0\\0&M_1/N_S\end{pmatrix}$$ has the following properties:
\begin{enumerate}
\item $\sigma_p \in \GL_2(\Z_p)$,
\item $c_p(\sigma_p)=n_p-c_p(\tau)$,
\end{enumerate} By strong approximation, we can pick $n \in N(\Q)$ such that \begin{enumerate}
 \item $n \equiv x_p \pmod{p^{n_p}}$ for all $p \in S$.
\item $n \in N(\Z_p)$ for all $p \notin S$.

\end{enumerate}

 We claim that this choice of $n$ has the required properties. Indeed, let $$\sigma= W \tau n \begin{pmatrix}1/M_1 &0\\0&M_1/N_S\end{pmatrix}.$$ Then our choice ensures that $\det(\sigma)=1$ and moreover Lemmas~\ref{keylemma1} and~\ref{keylemma2} ensure that for all primes $p$, we have
$\sigma \in \GL_2(\Z_p)$, $c_p(\sigma)=n_p-m_p$. It follows that $\sigma \in \SL_2(\Z)$ and $C(\sigma)=N/M$.
\end{proof}

\section{A gap principle}
 Our goal in this section is to prove the following Proposition.

\begin{proposition}\label{prop:gap}Let $z \in \H$. Then there exists a subset $S$ of $\mathcal{P}_N$, an Atkin-Lehner operator $W \in \mathcal{W}(S)$, an integer $M$ such that $M^2|N$, and an element $\sigma \in SL_2(\Z)$ such that the following are true.

\begin{enumerate}
\item $$C(\sigma)=N/M.$$

\item $$\Im(\sigma^{-1} W z) \ge \frac{\sqrt{3}M^2}{2N}.$$

\item For any $(0,0) \neq (c,d) \in \Z^2$, we have $$\left|c(\sigma^{-1} W z) + d\right|^2 \ge \frac{3M^2 (c, N/M^2) }{4N}.$$
\end{enumerate}
\end{proposition}

\begin{remark}In the above Proposition, we can always choose $\sigma$ to lie in some fixed set of representatives for $\Gamma_0(N)\bs \SL_2(\Z)/N(\Z)$. This is because both $\Im(\sigma^{-1} W z)$ and the set of possible values for $(c\sigma^{-1} W z + d)$ depend only on the class of $\sigma$ in $\SL_2(\Z)/N(\Z)$, while any product of $\Gamma_0(N)$ on the left of $\sigma$ can be absorbed into the $W$. In particular, if $N$ is squarefree, then $\sigma$ can be taken to equal the identity (in which case $M =1$ and our result essentially reduces to~\cite[Lemma 1]{harcos-templier-1}.)
\end{remark}
We begin with an elementary lemma.
\begin{lemma} Let $z_0 \in \H$ such that $\Im(z_0) \ge \frac{\sqrt{3}}{2}$. Then for all $n \in N(\R)$, $\gamma \in \SL_2(\Z)$, we have $\Im(z_0) \ge \frac{3}{4}\Im(\gamma n z_0)$.
\end{lemma}
\begin{proof}By replacing $z_0$ by $nz_0$ if necessary, we may assume that $n=1$. Also, by translating $z_0$ horizontally by an integer, we may assume that $\Re(z_0)$ lies between $-1/2$ and $1/2$. Now, the Lemma is immediate from the standard tiling of the upper half-plane by $\SL_2(\Z)-$translates of the standard fundamental domain for $\SL_2(\Z)$.\end{proof}

\begin{proof}[Proof of Proposition~\ref{prop:gap}]By the standard fundamental domain for $\SL_2(\Z)$, there exists $\tau \in \SL_2(\Z)$ such that $z=\tau z_0$ where $\Im(z_0)\ge \frac{\sqrt{3}}{2}.$ Now, let $\sigma, W, N_S, M,M_1$ be as in Proposition~\ref{keyprop1}. Then $C(\sigma)=N/M$. Note that $\frac{M_1^2}{N_S} \ge \frac{M^2}{N}$ (since $\frac{N_S}{M_1^2} | \frac{N}{M^2}$) and $\sigma^{-1} W \tau = \left(\begin{smallmatrix}M_1 &0\\0&N_S/M_1\end{smallmatrix}\right)n^{-1}$.  Furthermore

\begin{align*} \Im(\sigma^{-1} W z) &= \Im(\sigma^{-1} W \tau z_0) \\
&= \Im \left( \begin{pmatrix}M_1 &0\\0&N_S/M_1\end{pmatrix}n^{-1} z_0 \right) \\ &= (M_1^2/N_S) \Im(z_0)\\ &\ge \frac{M^2\Im(z_0)}{N}\\ &\ge \frac{\sqrt{3}M^2}{2N}.
\end{align*}

Next, given any pair $(c,d) \neq (0,0)$, we need to prove that $\left|c(\sigma^{-1} W z) + d\right|^2 \ge \frac{3M^2(c, N/M^2) }{4N}.$ It suffices to prove the result only in the case that $c$ and $d$ are coprime. Let $c_1 = c/(c, N_S/M_1^2)$ and $n_2 = N_S/(cM_1^2, N_S)$. Note that $c_1$ and $dn_2$ are coprime. Note also that $$\frac1{n_2} = \frac{(c,N_S/M_1^2)}{N_S/M_1^2} \ge \frac{(c, N/M^2)M^2}N .$$  Pick any $\gamma = \begin{pmatrix}a&b\\c_1&dn_2\end{pmatrix} \in \SL_2(\Z)$ and put $\gamma' = \begin{pmatrix}a&bM_1^2/N_S\\cn_2&dn_2\end{pmatrix} \in \SL_2(\Q)$. By the previous lemma, it follows that $\Im (z_0) \ge (3/4)\Im (\gamma n z_0)$ for all $n \in N(\R)$. Also, recall that $\Im(\sigma^{-1} W z) = (M_1^2/N_S) \Im(z_0).$

Then we have \begin{align*} \frac{M_1^2 \Im(z_0)}{n_2 N_S\left|c(\sigma^{-1} W z) + d\right|^2} &= \frac{\Im(\sigma^{-1} W z)}{n_2 \left|c(\sigma^{-1} W z) + d\right|^2} \\ &= \Im(\gamma'\sigma^{-1} W z)\\ &= \Im \left( \gamma'\begin{pmatrix}M_1 &0\\0&N_S/M_1\end{pmatrix}n^{-1} z_0 \right) \\
&= \frac{M_1^2}{N_S}\Im \left(  \gamma n^{-1} z_0   \right) \\ &\le  \frac{4M_1^2}{3N_S}\Im (z_0)
\end{align*}

giving us $$\left|c(\sigma^{-1} W z) + d\right|^2 \ge \frac3{4n_2} \ge \frac{3(c, N/M^2)M^2}{4N} ,$$
as desired.
\end{proof}

\section{Some counting results}

Let $1 \le N=N_2N_0^2$ with $N_2$ squarefree and let $M$ be a positive integer that divides $N_0$ (so $M^2 |N$). We define the region $G(N;M) \subset \H$ to consist of the points $z=x+iy \in \H$ with the following properties: \begin{enumerate}

  \item $y \ge \frac{\sqrt{3}M^2}{2N}.$

\item  For any pair of integers $(c,d) \neq (0,0)$, we have $\left|cz + d\right|^2 \ge \frac{3M^2 (c, N/M^2)}{4N}.$

    \end{enumerate}

For $z \in \H$, any $\delta>0$, and any integer $l \ge 1$, define:

 $$\Delta(l,N;M):=\left\{ \gamma=\begin{pmatrix}a&b\\c&d\end{pmatrix} \in M_2(\Z) : c \equiv 0 \bmod{N}, a \equiv 1 \bmod{M}, \det(\gamma)= l \right\}.$$

 $$N_*(z,l,\delta,N;M):=\left|\{ \gamma=\begin{pmatrix}a&b\\c&d\end{pmatrix} \in  \Delta(l,N;M) :  u(\gamma z, z) \le \delta, c \neq 0, (a+d)^2 \neq 4l \} \right|.$$

 $$N_u(z,l,\delta,N;M):=\left|\{ \gamma=\begin{pmatrix}a&b\\c&d\end{pmatrix} \in  \Delta(l,N;M) :  u(\gamma z, z) \le \delta, c = 0, (a+d)^2 \neq 4l \} \right|.$$

 $$N_p(z,l,\delta,N;M):=\left|\{ \gamma=\begin{pmatrix}a&b\\c&d\end{pmatrix} \in  \Delta(l,N;M) :  u(\gamma z, z) \le \delta,  (a+d)^2 =4l \} \right|.$$

$$N(z,l,\delta,N;M) :=N_*(z,l,\delta,N;M)+ N_u(z,l,\delta,N;M) + N_p(z,l,\delta,N;M).$$

\begin{remark}These definitions are similar to ones in~\cite{harcos-templier-2} except that we have the added condition $a \equiv 1 \bmod{M}.$
\end{remark}

In the sequel, we will estimate the above quantities, ultimately proving a result (Proposition~\ref{prop:ampl}) which will be useful for the amplification method to be used later in this paper. For the convenience of the reader, we begin by quoting a result that will be frequently used in this section.

\begin{lemma}[Lemma 2 of \cite{schmidt}]\label{countinglemma}Let $L$ be a lattice in $\R^2$ and $D\subset \R^2$ be a disc of radius $R$. If $\lambda_1$ is the distance from the origin of the shortest vector in $L$, and $d$ is the covolume of $L$, then $$|L \cap D| \ll 1 + \frac{R}{\lambda_1} + \frac{R^2}{d}.$$

\end{lemma}

The next lemma, which counts general matrices, is a mild generalization of~\cite[Lemmas 2.2 and 2.3]{harcos-templier-2}.

\begin{lemma}\label{l:general}Let $z \in G(N;M)$ and $M^2 \le L \le N^{O(1)}$. Let $1 \le l_1 \equiv 1 \pmod{M}$, $l_1 \le N^{O(1)}$. Then the following hold.

\begin{equation}\label{eq1}\sum_{\substack{1\le l \le L \\ l \equiv 1 \bmod M}} N_*(z,l, N^{\epsilon}, N;M) \ll_{\epsilon} N^{O(\epsilon)} \left(  \frac{L}{MNy} + \frac{L^{3/2}}{M^2 \sqrt{N}}  +\frac{L^2}{M^2N} \right).\end{equation}
 \begin{equation}\label{eq2}\sum_{\substack{1\le l \le L \\ l \equiv 1 \bmod M}} N_*(z,l^2, N^\epsilon, N;M) \ll_{\epsilon} N^{O(\epsilon)} \left(   \frac{L}{Ny} + \frac{L^2}{M\sqrt{N}} + \frac{L^{3}}{MN}  \right).\end{equation}
\begin{equation}\label{eq3}\sum_{\substack{1\le l_2 \le L \\ l_2 \equiv 1 \bmod M}} N_*(z,l_1l_2^2, N^\epsilon, N;M) \ll_{\epsilon} N^{O(\epsilon)} \left(  \frac{L^{3/2}}{Ny} + \frac{L^3}{M\sqrt{N}} + \frac{L^{9/2}}{MN}  \right).\end{equation}

\end{lemma}
\begin{proof}Let $\gamma=\begin{pmatrix}a&b\\c&d\end{pmatrix}$ satisfy the conditions $\gamma \in  \Delta(l,N;M)$, $  u(\gamma z, z) \le \delta$,  $c \neq 0$, $(a+d)^2 \neq 4l$ for some $1 \le l\le L$, $l \equiv 1 \bmod{M}.$ As in ~\cite{harcos-templier-2}, we conclude that there are $\ll_\epsilon N^{O(\epsilon)} L^{\frac12}/(Ny)$ possible values for $c$ and that the following inequality is satisfied: $$|-cz^2 + (a-d)z + b|^2 \le Ly^2 N^\epsilon.$$ We note that $(a-d) \equiv 0 \pmod{M}$. Putting $t=(a-d)/M$, and applying Lemma~\ref{countinglemma} to the lattice $\langle 1, Mz \rangle$ (note that $R = \sqrt{L}yN^{\epsilon/2}$, $d=My$ and $\lambda_1^2 \gg M^2/N$) we conclude that for each $c$, the number of pairs $(t, b)$ satisfying the above inequality is $\ll_\epsilon N^\epsilon (1 + \sqrt{LN}y/M + Ly/M).$ Moreover, as in~\cite{harcos-templier-2}, we conclude that $|a+d| \ll_\epsilon N^\epsilon L^\frac12.$ Since $a+d \equiv 2 \bmod{M}$, it follows that there are $ \ll_\epsilon N^\epsilon (L^\frac12/M)$ possibilities for $a+d$. This concludes the proof of~\eqref{eq1}.
Next, we prove~\eqref{eq2}. It suffices to show that
 $$\sum_{\substack{1\le l \le L \\ l =m^2 \\ m \equiv 1 \bmod M}} N_*(z,l, N^\epsilon, N;M) \ll_{\epsilon} N^{O(\epsilon)} \left(   \frac{L^{1/2}}{Ny} + \frac{L}{M\sqrt{N}} + \frac{L^{3/2}}{MN}  \right).$$
To prove this, we proceed exactly as in the previous case, except that we deal with the number of possibilities for $a+d$ differently. Indeed, we have the equation $$(a+d - 2m)(a+d+2m) = (a-d)^2 +4bc.$$ Hence, given $c$, $b$, $a-d$, there are $\ll_\epsilon N^\epsilon$ pairs $(a+d, m)$ satisfying the given constraints. This proves the desired bound.

Finally, we deal with~\eqref{eq3}. Once again, we proceed as in the first case (note that now we have $R = L^{3/2}yN^{\epsilon/2}$, $d=My$ and $\lambda_1^2 \gg M^2/N$). However this time we deal with $a+d$ differently, namely via the observation that the pair $(a+d, l_2)$ satisfies a generalized Pell's equation. As the details are identical to~\cite[Lemma 2.3]{harcos-templier-2}, we omit them.\footnote{In fact, it turns out that we can completely avoid dealing with this case by making a small adjustment in the proof of our main theorem; see Remark~\ref{r:amplifier}.}
\end{proof}

Next we count upper-triangular matrices. The
lemma below is a mild generalization of~\cite[Lemma 2.4]{harcos-templier-2}

\begin{lemma}Let $z \in G(N;M)$ and $M\le L \le N^{O(1)}$. Then the following estimates hold. \begin{equation}\label{eq4}\sum_{\substack{1\le l_1, l_2 \le L \\ l_1 \equiv l_2 \equiv 1 \bmod M \\ l_1, l_2 \text{ are primes }}} N_u(z,l_1l_2, N^\epsilon, N;M) \ll_{\epsilon} N^{O(\epsilon)} \left(\frac{L}{M} +    \frac{L^2y\sqrt{N}}{M^2} + \frac{L^3y}{M^2}   \right).\end{equation}

\begin{equation}\label{eq5}\sum_{\substack{1\le l_1, l_2 \le L \\ l_1 \equiv l_2 \equiv 1 \bmod M\\ l_1, l_2 \text{ are primes }}} N_u(z,l_1l_2^2, N^\epsilon, N;M) \ll_{\epsilon} N^{O(\epsilon)} \left( \frac{L}{M} +   \frac{L^{5/2}y\sqrt{N}}{M^2} + \frac{L^4y}{M^2}   \right).\end{equation}
\begin{equation}\label{eq6}\sum_{\substack{1\le l_1, l_2 \le L \\ l_1 \equiv l_2 \equiv 1 \bmod M\\ l_1, l_2 \text{ are primes }}} N_u(z,l_1^2l_2^2, N^\epsilon, N;M) \ll_{\epsilon} N^{O(\epsilon)} \left( 1 +   \frac{L^2 y \sqrt{N}}{M} + \frac{L^4y}{M}   \right).\end{equation}
\begin{equation}\label{eq7}\sum_{\substack{1\le l_1 \le L \\ l_1 \equiv 1 \bmod M}} N_u(z,l_1, N^\epsilon, N;M) \ll_{\epsilon} N^{O(\epsilon)} \left( 1 +   \frac{L^{1/2} y \sqrt{N}}{M} + \frac{Ly}{M}   \right).\end{equation}

\end{lemma}
\begin{proof}Let $\gamma=\begin{pmatrix}a&b\\0&d\end{pmatrix}$ satisfy the conditions $\gamma \in  \Delta(l,N;M)$, $u(\gamma z, z) \le \delta$, $(a+d)^2 \neq 4l$ for some $1 \le l \le \Lambda, \ l \equiv 1 \bmod{M}.$ As in ~\cite{harcos-templier-2}, we conclude that the following inequality is satisfied: $$|(a-d)z + b|^2 \le \Lambda y^2 N^\epsilon.$$ Also, note that $(a-d) \equiv 0 \pmod{M}$. The rest of the proof is identical to the proof of~\cite[Lemma 2.4]{harcos-templier-2}, with the modifications for $M$ as in Lemma~\ref{l:general} above.
\end{proof}

Finally we count parabolic matrices. The next lemma is a significant extension of~\cite[Lemma 2]{harcos-templier-1}.
\begin{lemma}\label{l:para}Let $z \in G(N;M)$ and $1 \le l \equiv 1 \pmod{M}$. Then  \begin{enumerate}

\item $N_p(z,l, N^\epsilon, N;M) = 0$ if $l$ is not a perfect square.

\item Suppose that $l=m^2$ with $m>0$ an integer. Suppose also that both $l$ and $y$ are $\le N^{O(1)}$. Then $$N_p(z,l, N^\epsilon, N;M) \ll_\epsilon 1 + N^{O(\epsilon)} \left( \frac{myN_0}{ M}  + \frac{mN_0}{N}\right).$$

\end{enumerate}
\end{lemma}
\begin{proof} Let $\gamma \in \Delta(l,N;M)$ be  such that $u(\gamma z, z) \le N^\epsilon$ and $\trace(\gamma)^2 = 4l$. Then $\gamma$ fixes some point $\tau(\infty)$ where $\tau \in \SL_2(\Z).$ Hence $\gamma' = \tau^{-1}\gamma \tau$ fixes the point $\infty$, and hence is a parabolic upper-triangular matrix with integer coefficients and determinant $l$. It follows that $l$ must be a perfect square. Writing $\gamma' = \pm \begin{pmatrix}m&t\\0&m\end{pmatrix}$ (where $m^2= l$ and $ t \in \Z$) and $\tau^{-1} = \begin{pmatrix}a&b\\c&d\end{pmatrix}$, we see that $\gamma = \pm \begin{pmatrix}m + cdt & d^2t\\ -c^2t &m - cdt    \end{pmatrix}.$  This shows that $N | c^2t$. Moreover
$u(\gamma z, z) = u(\gamma'z', z') $ where $z'=\tau^{-1}z$. Writing $z'=x'+iy'$, we note that \begin{equation}\label{paraeq}N^\epsilon \ge u(\gamma'z', z') = \frac{t^2}{4ly'^2} = \frac{t^2|cz+d|^4}{4ly^2} \gg \frac{t^2 (c, N/M^2)^2 M^4}{ly^2 N^2}.\end{equation}

Next, if $t=0$ then $\gamma' = \gamma = \pm \begin{pmatrix}m&0\\0&m\end{pmatrix}$  is the only possibility. So it suffices to consider the case $t>0$. We claim that if $t >0$ then $\frac{t^2 M^4 (c, N/M^2)^2}{N^2} \ge \frac{t_0^2M^2}{N_0^2} $ where $t_0 = t/(t, N^\infty)$ is the $N$-free part of $t$. Let $p$ divide $N$, with $p^{t'} || t$, $p^{c'}||c$, $p^{n_p}||N $, $p^{m_p}||M$. Then, it suffices to show that $$ \min(2t' +2c'+ 4m_p,2t' + 2n_p) \ge   2n_p + 2m_p - 2\lfloor \frac{n_p}{2} \rfloor.$$ If $2c'+ 4m_p + 2t' \ge 2n_p$ this is immediate (note that $m_p \le \lfloor \frac{n_p}{2} \rfloor$). So we assume that $2c'+ 4m_p + 2t' < 2n_p$.  It suffices in this case to prove that $2t' +2c' \ge 2n_p  - 2\lfloor \frac{n_p}{2} \rfloor. $ But this follows immediately from the fact that $2t' +2c' \ge n_p$ and because $2t' +2c'$ is even.

 So we have proved that if $t \neq 0$ then $$N^\epsilon \ge \frac{t^2 (c, N/M^2)^2 M^4}{ly^2 N^2} \ge \frac{t_0^2M^2}{ly^2N_0^2}.$$
Hence $$t_0 \le N^\epsilon \frac{yN_0 \sqrt{l}}{M}.$$ On the other hand~\eqref{paraeq} implies that $$t \le N^\epsilon \frac{y \sqrt{l}}{|cz+d|^2} \ll \frac{y N\sqrt{l}}{M^2}.$$ Write $t=t_0t_1$ where $t_1 |N^\infty$. Given any such $t=t_0t_1$, let us count the number of admissible $\gamma$; this reduces to counting the number of admissible $c,d$.   Given any integer $f = \prod p_i^{a_i}$, we define (temporarily) $\{\sqrt{f} \} = \prod p_i^{\lceil a_i/2 \rceil}.$ Note that if $f$ divides $a^2$ for some integer $a$, then $\{\sqrt{f} \}$ divides $a$. Note also that $\{\sqrt{N} \} = N_2N_0 = N/N_0.$ We have already proved that $ N |c^2t_1$. It follows that $\{ \sqrt{g} \}$ divides $c$ where $g=N/(t_1, N)$. Note also that $\{ \sqrt{g} \} \ge \frac{N_2N_0}{(t_1, N)} \ge \frac{N}{N_0t_1}.$  Let us count the number of pairs of integers $c,d$ such that $|cz + d|^2 \ll N^\epsilon \frac{y \sqrt{l}}{t}$ and $\{ \sqrt{g} \}$ divides $c$. Considering the lattice $\langle 1, \{ \sqrt{g} \}z \rangle$ we see that the quantity $\lambda_1$ for this lattice satisfies $\lambda_1^2 \ge \frac{(\{ \sqrt{g} \}M^2,N)}{N} \ge \frac{\{ \sqrt{g} \}M}{N} \ge \frac{M}{N_0t_1}$ (we used here the fact that $\{ \sqrt{g} \}M$ divides $N$, which follows as $\{ \sqrt{g} \}$ divides $N_2N_0$ and $M$ divides $N_0$). Furthermore, the covolume $d$ of this lattice satisfies $ d \ge \frac{N y}{N_0t_1}.$ Hence, using Lemma~\ref{countinglemma},   the total number of admissible $c,d$ for each fixed $t=t_0t_1$ is $$\ll_\epsilon 1 + \frac{N^\epsilon ((N_0t_1)^{1/2} y^{1/2}l^{1/4}}{t^{1/2} \sqrt{M}} + \frac{N^\epsilon  \sqrt{l} N_0}{t_0 N }.$$ Hence, the total number of parabolic matrices $\gamma \in \Delta(l,N;M)$ such that $u(\gamma z, z) \le N^\epsilon$ is $$\ll_\epsilon 1 + N^\epsilon \sum_{\substack{1 \le t_0 \le \frac{yN_0 \sqrt{l}}{M} \\ (t_0,N)=1}} \sum_{\substack{1 \le t_1 \le  \frac{yN \sqrt{l}}{M^2t_0}\\ t_1|N^\infty}}\left(1 + \frac{(N_0ym)^{1/2}}{(t_0M)^{1/2}} + \frac{m N_0}{t_0 N } \right)$$ $$ \ll_\epsilon 1 + N^{O(\epsilon)} \left(\frac{yN_0m}{M} + \frac{m N_0}{ N } \right)$$
In the last step above, we used a fact that will also be used a few times later in this paper: for all positive integers $X$, $N$, one has the bound  $\sum_{\substack{t_1 \le X \\ t_1 | N^\infty}} 1 \ll_\epsilon (NX)^\epsilon.$ The proof of this fact follows from Rankin's trick.\footnote{Observe that $\sum_{\substack{t_1 \le X \\ t_1 | N^\infty}} 1 \le X^\epsilon \prod_{t_1|N^\infty} t_1^{-\epsilon} = X^\epsilon \prod_{p|N} (1-p^{-\epsilon})^{-1}$ and then apply the divisor bound. }
\end{proof}

Combining all the above bounds, we get the following proposition, which is all that we will use later.

\begin{proposition}\label{prop:ampl}Let $1 \le N=N_2N_0^2$ with $N_2$ squarefree and let $M$ be a positive integer that divides $N_0$. Suppose that $z = x+iy \in G(N;M)$ and assume further that $y \le N^{-1/2}$. Let $M^2 \le \Lambda \le N^{O(1)}$.

Define $$y_l := \begin{cases} \frac{\Lambda}{M}, & l=1, \\ 1, & l\in\{l_1, l_1l_2, l_1l_2^2, l_1^2l_2^2\} \text{ with }\Lambda <l_1, l_2 < 2\Lambda \text{ primes, } l_1 \equiv l_2 \equiv 1 \mod M,\\ 0 & \text{ otherwise.}\end{cases}$$

Then \begin{equation}\label{amplesum}\sum_{l\ge 1}\frac{y_l}{\sqrt{l}} N(z,l, N^\epsilon, N;M) \ll_\epsilon N^{O(\epsilon)}  \left( \frac{\Lambda}{M} + \frac{\Lambda^2yN_0}{M^3}+ \frac{\Lambda^{5/2}}{M^2\sqrt{N}} + \frac{\Lambda^4}{MN} \right).\end{equation}

\end{proposition}
\begin{proof}The contribution to the LHS of~\eqref{amplesum} from the parabolic matrices is $$\ll_\epsilon \frac{N^{O(\epsilon)} }{M^2}\left(1 + \frac{\Lambda^2 y N_0}{M} + \frac{\Lambda^2\sqrt{N}}{N} \right) \ll_\epsilon N^{O(\epsilon)}  \left(\frac{\Lambda}{M} + \frac{\Lambda^2yN_0}{M^3}+ \frac{\Lambda^{5/2}}{M^2\sqrt{N}} \right)$$

using Lemma~\ref{l:para} and $\Lambda \ge M$.

The contribution to the LHS of~\eqref{amplesum} from the upper-triangular matrices with $l=1$ is $$\ll_\epsilon \frac{N^{O(\epsilon)}  \Lambda}{M} \left(1 + y \frac{\sqrt{N}}{\sqrt{M}} \right) \ll_\epsilon N^{O(\epsilon)}  \left(\frac{\Lambda}{M} \right)$$
using~\eqref{eq7}. For $\Lambda < l < 2\Lambda$, it is $$\ll_\epsilon \frac{N^{O(\epsilon)}  }{\sqrt{\Lambda}} \left(1 + y \frac{\sqrt{\Lambda N}}{M} + \frac{\Lambda y}{M} \right) \ll_\epsilon N^{O(\epsilon)}  \left(\frac{\Lambda}{M} \right)$$
using~\eqref{eq7}. For $\Lambda^2 < l < 4\Lambda^2$, it is $$\ll_\epsilon \frac{N^{O(\epsilon)}  }{\Lambda} \left(\frac{\Lambda}{M} +  \frac{\Lambda^2 y\sqrt{ N}}{M^2} + \frac{\Lambda^3 y}{M^2} \right) \ll_\epsilon N^{O(\epsilon)}  \left(\frac{\Lambda}{M} +  \frac{\Lambda^{5/2}}{M^2\sqrt{N}}\right)$$
using~\eqref{eq4}. For $\Lambda^3 < l < 8\Lambda^3$, it is $$\ll_\epsilon \frac{N^{O(\epsilon)}  }{\Lambda^{3/2}} \left(\frac{\Lambda}{M} +  \frac{\Lambda^{5/2} y\sqrt{ N}}{M^2} + \frac{\Lambda^4 y}{M^2} \right) \ll_\epsilon N^{O(\epsilon)}  \left(\frac{\Lambda}{M} +  \frac{\Lambda^{5/2}}{M^2\sqrt{N}}\right)$$
using~\eqref{eq5}. For $\Lambda^4 < l $, it is $$\ll_\epsilon \frac{N^{O(\epsilon)}  }{\Lambda^{2}} \left( 1 +  \frac{\Lambda^{2} y\sqrt{ N}}{M} + \frac{\Lambda^4 y}{M} \right) \ll_\epsilon N^{O(\epsilon)}  \left(\frac{\Lambda}{M} +  \frac{\Lambda^{5/2}}{M^2\sqrt{N}}\right)$$
using~\eqref{eq6}.

The contribution to the LHS of~\eqref{amplesum} from the general matrices with $l=1$ is $$\ll_\epsilon \frac{N^{O(\epsilon)}  \Lambda}{M} \left(\frac{M}{Ny} +  \frac{M}{\sqrt{N}} +  \frac{M^2}{N}\right) \ll_\epsilon N^{O(\epsilon)}  \left(\frac{\Lambda}{M} \right)$$
using~\eqref{eq1}. For $\Lambda < l < 2\Lambda$, it is $$\ll_\epsilon \frac{N^{O(\epsilon)}  }{\sqrt{\Lambda}} \left(\frac{\Lambda}{MNy} +  \frac{\Lambda^{3/2}}{M^2\sqrt{N}} + \frac{\Lambda^2}{M^2N} \right) \ll_\epsilon N^{O(\epsilon)}  \left(\frac{\Lambda}{M} + \frac{\Lambda^{5/2}}{M^2\sqrt{N}}  \right)$$
using~\eqref{eq1}. For $\Lambda^2 < l < 4\Lambda^2$, it is $$\ll_\epsilon \frac{N^{O(\epsilon)}  }{\Lambda} \left(\frac{\Lambda^2}{MNy} +  \frac{\Lambda^3}{M^2\sqrt{N}} + \frac{\Lambda^4}{M^2N} \right) \ll_\epsilon N^{O(\epsilon)}  \left(\frac{\Lambda}{M} +  \frac{\Lambda^{5/2}}{M^2\sqrt{N}} + \frac{\Lambda^4}{MN}\right)$$
using~\eqref{eq1}. For $\Lambda^3 < l < 8\Lambda^3$, it is $$\ll_\epsilon \frac{N^{O(\epsilon)}  \Lambda}{\Lambda^{3/2} M} \left(\frac{\Lambda^{3/2}}{Ny} +  \frac{\Lambda^3}{M\sqrt{N}} + \frac{\Lambda^{9/2}}{MN} \right) \ll_\epsilon N^{O(\epsilon)}  \left(\frac{\Lambda}{M} +  \frac{\Lambda^{5/2}}{M^2\sqrt{N}} +  \frac{\Lambda^4}{MN}\right)$$
using~\eqref{eq3}. For $\Lambda^4 < l $, it is $$\ll_\epsilon \frac{N^{O(\epsilon)}  }{\Lambda^{2}} \left(  \frac{\Lambda^{2}}{Ny} + \frac{\Lambda^4 }{M\sqrt{N}} + \frac{\Lambda^6}{MN}\right) \ll_\epsilon N^{O(\epsilon)}  \left(\frac{\Lambda}{M} +  \frac{\Lambda^{5/2}}{M^2\sqrt{N}} + \frac{\Lambda^4}{MN}\right)$$
using~\eqref{eq2}.

The proof is complete.
\end{proof}

\begin{remark}\label{r:amplifier}Gergely Harcos and Guillaume Ricotta have pointed out to the author the possibility of using an improved amplifier, as in~\cite{blomer-harcos-milicevic}, in the proof of our main result. With this modification, we would only need to prove a weaker version of the above Proposition, where the terms corresponding to $l=l_1$ or $l_1l_2^2$ are removed.

\end{remark}

\section{Hecke operators on $\Gamma_0(N;M)$}
We begin by recalling the usual Hecke algebra on $\Gamma_0(N)$. Define $$\Delta_0(N):= \left\{ \gamma=\begin{pmatrix}a&b\\c&d\end{pmatrix} \in M_2(\Z) : c \equiv 0 \bmod{N}, \ \det(\gamma)>0 \right\},$$ $$\mathcal{H}_0(N):= \{\sum_{\alpha \in \Delta_0(N)}t_\alpha \Gamma_0(N)\alpha\Gamma_0(N): t_\alpha \in \Z, \ t_\alpha=0 \text{ for almost all }\alpha \}.$$

Next, for any divisor $M$ of $N$, define $$\Delta_0(N;M):= \left\{ \gamma=\begin{pmatrix}a&b\\c&d\end{pmatrix} \in M_2(\Z) : c \equiv 0 \bmod{N}, a \equiv 1 \bmod{M}, \ \det(\gamma)>0 \right\},$$  $$\mathcal{H}_0(N;M):= \{\sum_{\alpha \in \Delta_0(N;M)}t_\alpha \Gamma_0(N;M)\alpha\Gamma_0(N;M): t_\alpha \in \Z, \ t_\alpha=0 \text{ for almost all }\alpha \}.$$

Elements of $\mathcal{H}_0(N)$ (resp. $\mathcal{H}_0(N;M)$) act on Maass cusp forms $f$ on the group $\Gamma_0(N)$ (resp. $\Gamma_0(N;M)$) in the usual manner. We will normalize this action as follows: $$f|\Gamma\alpha\Gamma = \det(\alpha)^{-1/2} \sum_{\gamma \in \Gamma \bs \Gamma\alpha\Gamma}f|\gamma, \qquad \text{ where $\Gamma= \Gamma_0(N)$ or $ \Gamma_0(N;M)$. }$$

Consider the natural map from $\mathcal{H}_0(N;M)$ to $\mathcal{H}_0(N)$ defined via $$\Gamma_0(N;M)\alpha\Gamma_0(N;M) \mapsto \Gamma_0(N)\alpha\Gamma_0(N).$$ Standard arguments (see the remarks above Thm. 4.5.19 of~\cite{Miyake}) imply that this map is an \emph{isomorphism} of Hecke algebras. For $T \in \mathcal{H}_0(N)$, let $T'$ denote its image in $\mathcal{H}_0(N;M)$ under this isomorphism. Then, given any Maass cusp form $f$ on the group $\Gamma_0(N)$, (which therefore can also be thought of as a cusp form on the group $\Gamma_0(N;M)$) and any $T \in \mathcal{H}_0(N)$,  one has the compatibility relation $$f|T = f|T'.$$

For any integer $l \ge 1$, we let $T(l) \in  \mathcal{H}_0(N)$ be the $\Z$-linear span of the double cosets $\Gamma_0(N)\alpha\Gamma_0(N)$ for $\alpha \in \Delta_0(N)$ of determinant $l$. Then, since the above-defined isomorphism of Hecke algebras is determinant-preserving on double cosets, it follows that $T(l)'$ is the $\Z$-linear span of the double cosets $\Gamma_0(N)\alpha\Gamma_0(N)$ where $\alpha \in \Delta_0(N;M)$ has determinant $l$. Recall the definition from the previous section  $$\Delta(l,N;M):=\left\{ \gamma=\begin{pmatrix}a&b\\c&d\end{pmatrix} \in M_2(\Z) : c \equiv 0 \bmod{N}, a \equiv 1 \bmod{M}, \det(\gamma)= l \right\}.$$ The comments above imply that if $f$ is a Maass cusp form on the group $\Gamma_0(N)$ that is an eigenform for the Hecke operator $T(l)$ with (normalized) eigenvalue $\lambda_f(l)$, then \begin{equation}\label{eq:compatible}\sum_{\gamma \in \Gamma_0(N;M)  \bs \Delta(l,N;M) }f|\gamma = l^{1/2} \lambda_f(l) f    \end{equation}

Now, suppose that $M^2$ divides $N$, and $\sigma \in \SL_2(\Z)$ satisfies $C(\sigma) = N/M$. Then, by Proposition~\ref{gammanmcompatibility}, we know that the map $g \mapsto g|\sigma$ is an \emph{endomorphism} of the space of Maass cusp forms for the group $\Gamma_0(N;M)$. It is a natural question if this endomorphism commutes with the Hecke algebra action on the same space. While this is not true in general, it is indeed true for the Hecke operators $T(l)'$ for $l \equiv 1 \pmod{M}$.

\begin{proposition}\label{prop:heckecompatible}Let $M$ be a positive integer such that $M^2$ divides $N$ and let $\sigma \in \SL_2(\Z)$ satisfy $C(\sigma) = N/M$. Let $g$ be a Maass cusp form for the group $\Gamma_0(N;M)$. Then, for any positive integer $l$ such that $l \equiv 1 \pmod{M}$, we have the relation $$g|T(l)'|\sigma = g|\sigma|T(l)'.$$
\end{proposition}
\begin{proof} Recall that for any Maass cusp form $h$ for the group $\Gamma_0(N;M)$, we have $$h|T(l)' = l^{-1/2} \sum_{\gamma \in \Gamma_0(N;M)  \bs \Delta(l,N;M) }h|\gamma.$$ So it suffices to prove that $$\sigma \Delta(l,N;M) \sigma^{-1} = \Delta(l,N;M) .$$ But this follows from equation~\eqref{eq:keygammanm}.
\end{proof}

This gives us the following corollary, which is all that we will use  in the sequel.

\begin{corollary}\label{heckecor}Let $f$ be a Hecke-Maass cuspidal newform for the group $\Gamma_0(N)$. For any $n \ge 1$, let $\lambda_f(n)$ denote the (normalized) $n$th Hecke eigenvalue for $f$. Let $M$ be a positive integer such that $M^2$ divides $N$ , $\sigma \in \SL_2(\Z)$ satisfy $C(\sigma) = N/M$, and $l$ be a positive integer such that $l \equiv 1 \pmod{M}$. Then, if $g:= f| \sigma$, then $$\sum_{\gamma \in \Gamma_0(N;M)  \bs \Delta(l,N;M) }g|\gamma = \lambda_f(l) l^{1/2} g.$$
\end{corollary}
\begin{proof}This follows by combining~\eqref{eq:compatible} and Proposition~\ref{prop:heckecompatible}.
\end{proof}
\begin{remark}The results of this section continue to hold in the holomorphic  case.
\end{remark}

\begin{remark}The methods and proofs of this section are similar in spirit to those in Section 3.5 of Shimura's book~\cite{MR0314766}.
\end{remark}
\section{The bound via Fourier expansions at width 1 cusps}
We will prove the following Proposition.
 \begin{proposition}\label{prop:fourier}Let $f$ be a Hecke-Maass cuspidal newform for the group $\Gamma_0(N)$ with Laplace eigenvalue $\lambda = 1/4 + r^2$. Let $M \ge 1$ be an integer such that $M^2 |N$ and let $\sigma \in \SL_2(\Z)$ satisfy $C(\sigma)= N/M$. Assume further that $\langle f, f \rangle_{\Gamma_0(N)} =1$ and $|r| \le R$. Then we have:

$$|f(\sigma z)| = |(f|\sigma)(z)| \ll_{R,\epsilon}  N^\epsilon \cdot \begin{cases}\frac{1}{(Ny)^{1/2}}, & \frac1N \le y \le \frac1{M^2}, \\ \frac{M^{1/2}}{N^{1/2}y^{1/4}}, &  \frac1{M^2} \le y .  \end{cases}$$

 \end{proposition}

The proof will follow from a careful analysis of the Fourier expansion at the cusp $\sigma(\infty)$. Let $f$ be as in the Proposition. Then $f$ has the usual Fourier expansion at $\infty$,
$$f(z)= y^{1/2} \sum_{n \ne 0} \rho(n)K_{ir}(2 \pi |n| y) e(nx).$$ We have the equation $|\lambda_f(|n|)| = |\rho(n)/\rho(1)|$, where  for each $l \ge 1$, $\lambda_f(l)$ denotes the (normalized) $l$th Hecke eigenvalue for $f$. Let $\sigma \in \SL_2(\Z)$ satisfy $C(\sigma)= N/M$ (so $W(\sigma)=1$) and let $h = f|\sigma$. Then
  $h$ is a Maass cusp form for the congruence subgroup $\Gamma_0(N;M)$. It has a Fourier expansion $$h(z) =y^{1/2} \sum_{n \ne 0} \rho_\sigma(n)K_{ir}(2 \pi |n| y) e(nx).$$ The coefficients $\rho_\sigma(n)$ are the Fourier coefficients of $f$ at the cusp $\sigma(\infty);$ unlike the coefficients at infinity, these cannot be understood simply in terms of Hecke eigenvalues (in fact, they are not even multiplicative). These coefficients were studied adelically in~\cite[Sec. 3.4.2]{NPS}, and we will use some calculations from there in what follows.\footnote{In~\cite[Sec. 3.4.2]{NPS}, we restricted ourselves to the holomorphic case but this does not matter because we will only use some local non-archimedean calculations from there which are the same for Maass and holomorphic forms.}

 The adelization of the form $f$ gives rise to a cuspidal automorphic representation $\pi=\otimes_{p\le \infty} \pi_p$ of $\GL_2(\A)$. Let $W= \prod_{p\le \infty}W_p$ be the global Whittaker newform in $\pi$ (with respect to the standard additive character $\psi = \prod_{p\le \infty}\psi_p$), where we normalize at the non-archimedean places so that $W_p(1)=1$ for all finite primes $p$. Fix an integer $a$ such that the cusp $\sigma(\infty)$ contains a representative of the form $\frac{a}{N/M}$ with $\gcd(a,N) =1$. For each integer $n$, define

 $$\lambda_{\sigma, N}(n) = \prod_{p|N}\left( (n, p^\infty)^{1/2} W_p\left(\begin{pmatrix}nM^2/N^2&0\\0&1 \end{pmatrix}\begin{pmatrix}0&1\\-1&0\end{pmatrix}
 \begin{pmatrix}1 & -aM/N\\0&1 \end{pmatrix} \right) \right).$$

 Using the usual adelic intepretation of Fourier coefficients as Whittaker functions, one observes (see the discussion following~\cite[(48)]{NPS}) that\footnote{A comparison with~\cite{NPS} reveals a conflict between~\eqref{e:fourierfactor} and~\cite[(49)]{NPS}. This reflects a typo in \cite{NPS}; the version stated here is correct.}

 \begin{equation}\label{e:fourierfactor}|\rho_\sigma(n)| = \left|\lambda_{\sigma, N}(n)\rho(1) \lambda_f\left( \frac{|n| }{(|n|, N^\infty) } \right)  \right|,\end{equation}

\begin{lemma}\label{l:firstbound}For all $X >0$, we have $$\sum_{0\le |n| \le X}|\lambda_{\sigma, N}(n)|^2 \ll_{\epsilon} (NX)^\epsilon (X + M\sqrt{X}).$$

\end{lemma}
\begin{proof} The proof is rather involved. Write
$n=n_0n_1$ where $n_1 := (n, N^\infty)$. Let us first show that\footnote{This fact was implicitly proved in~\cite{NPS} but we give a proof here for completeness.}  \begin{equation}\label{whitinv}|\lambda_{\sigma, N}(n_0n_1)| = |\lambda_{\sigma, N}(n_0'n_1)| \qquad \text{ if } (n_0, N) = (n_0',N) = 1, \ n_0 \equiv n_0' \pmod{M}.\end{equation} Indeed, to prove~\eqref{whitinv}, it suffices to show that for each $p|N$, $r \in \Z$, and $u_i \in \Z_p^\times$ with $u_1 \equiv u_2 \pmod{p^{m_p}}$, one has $$ \left|W_p\left(\begin{pmatrix}u_1 p^{r}&0\\0&1 \end{pmatrix}\begin{pmatrix}0&1\\-1&0\end{pmatrix}
 \begin{pmatrix}1 & -ap^{m_p-n_p}\\0&1 \end{pmatrix} \right) \right|=  \left|W_p\left(\begin{pmatrix}u_2 p^{r}&0\\0&1 \end{pmatrix}\begin{pmatrix}0&1\\-1&0\end{pmatrix}
 \begin{pmatrix}1 & -ap^{m_p-n_p}\\0&1 \end{pmatrix} \right)\right| .$$ Put $$\nu= \begin{pmatrix}1&p^{r-m_p+n_p}(u_1 - u_2)a^{-1}\\0&1\end{pmatrix}, \quad k=  \begin{pmatrix}u_1/u_2 &0\\p^{-m_p+n_p}(u_2 - u_1)a^{-1} &u_2/u_1 \end{pmatrix}.$$ Note that $k\in K_0(p^{n_p})$. We can check that $$\nu\begin{pmatrix} p^{r}&0\\0&1 \end{pmatrix}\begin{pmatrix}0&1\\-1&0\end{pmatrix}
 \begin{pmatrix}1 & -au_1^{-1}p^{m_p-n_p}\\0&1 \end{pmatrix} = \begin{pmatrix} p^{r}&0\\0&1 \end{pmatrix}\begin{pmatrix}0&1\\-1&0\end{pmatrix}
 \begin{pmatrix}1 & -au_2^{-1}p^{m_p-n_p}\\0&1 \end{pmatrix} k.$$Now~\eqref{whitinv} follows from the following calculation: \begin{align*}W_p\left(\begin{pmatrix}u_1 p^{r}&0\\0&1 \end{pmatrix}\begin{pmatrix}0&1\\-1&0\end{pmatrix}
 \begin{pmatrix}1 & -ap^{m_p-n_p}\\0&1 \end{pmatrix} \right) &= W_p\left(\begin{pmatrix} p^{r}&0\\0&1 \end{pmatrix}\begin{pmatrix}0&1\\-1&0\end{pmatrix}
 \begin{pmatrix}1 & -au_1^{-1}p^{m_p-n_p}\\0&1 \end{pmatrix} \right) \\&= \epsilon  W_p\left(\nu\begin{pmatrix} p^{r}&0\\0&1 \end{pmatrix}\begin{pmatrix}0&1\\-1&0\end{pmatrix}
 \begin{pmatrix}1 & -au_1^{-1}p^{m_p-n_p}\\0&1 \end{pmatrix} \right) \\&= \epsilon  W_p\left(\begin{pmatrix} p^{r}&0\\0&1 \end{pmatrix}\begin{pmatrix}0&1\\-1&0\end{pmatrix}
 \begin{pmatrix}1 & -au_2^{-1}p^{m_p-n_p}\\0&1 \end{pmatrix} k \right) \\ &=\epsilon  W_p\left(\begin{pmatrix}u_2 p^{r}&0\\0&1 \end{pmatrix}\begin{pmatrix}0&1\\-1&0\end{pmatrix}
 \begin{pmatrix}1 & -ap^{m_p-n_p}\\0&1 \end{pmatrix} \right)
 \end{align*}
where $\epsilon = \psi_p(p^{r-m_p+n_p}(u_1 - u_2)a^{-1}) \in S^{1}$.

So, for each $n_1 |N^\infty$, we can define the quantity $$ \lambda_{[N/M], N}(n_1):=   \left(\frac{1}{\phi(M)}\sum_{\substack{n_0 \bmod{M} \\ (n_0, M) =(n_0, N)=1}} |\lambda_{\sigma, N}(n_0n_1)|^2\right)^{1/2},$$ where $\sum_{\substack{n_0 \bmod{M} \\ (n_0, M) = (n_0, N)=1}}$ means that the sum is taken over any set of integers $n_0$ which form a reduced residue system modulo $M$ and such that each $n_0$ is coprime to $N$ (e.g., if $M=5$, $N=50$, we can sum over the elements $1,3,7,9$).

Next, note that \begin{align*}\lambda_{[N/M], N}(n_1)^2&=   \frac{1}{\phi(M)}\sum_{\substack{n_0 \bmod{M} \\ (n_0, M) =(n_0, N)=1}} |\lambda_{\sigma, N}(n_0n_1)|^2\\&=   \frac{1}{\phi(N)}\sum_{\substack{n_0 \bmod{N} \\ (n_0, N)=1}} |\lambda_{\sigma, N}(n_0n_1)|^2\\&=   \frac{n_1^{1/2}}{\phi(N)}\sum_{\substack{n_0 \bmod{N} \\ (n_0, N)=1}} \prod_{p|N} \left|W_p\left(\begin{pmatrix}n_0n_1M^2/N^2&0\\0&1 \end{pmatrix}\begin{pmatrix}0&1\\-1&0\end{pmatrix}
 \begin{pmatrix}1 & -aM/N\\0&1 \end{pmatrix} \right) \right|^2\\&=   \frac{n_1^{1/2}}{\phi(N)}\sum_{\substack{n_0 \bmod{N} \\ (n_0, N)=1}} \prod_{p|N} \left|W_p\left(\begin{pmatrix}-a^{-1}n_0n_1M^2/N^2&0\\0&1 \end{pmatrix}\begin{pmatrix}0&1\\-1&0\end{pmatrix}
 \begin{pmatrix}1 & M/N\\0&1 \end{pmatrix} \right) \right|^2\\&= \prod_{p|N} (n_1, p^\infty)\int_{n_0 \in \Z_p^\times} \left|W_p\left(\begin{pmatrix}n_0n_1M^2/N^2&0\\0&1 \end{pmatrix}\begin{pmatrix}0&1\\-1&0\end{pmatrix}
 \begin{pmatrix}1 & M/N\\0&1 \end{pmatrix} \right) \right|^2
 \end{align*}

 where the last step follows from the invariance properties of $W_p$ and an application of the Chinese Remainder Theorem.

   This shows that the quantity $\lambda_{[N/M], N}(n_1)$ factors as $$\lambda_{[N/M], N}(n_1) = \prod_{p|N} \lambda_{[N/M], p}(n_1),$$ where for each integer $n_1$, each prime $p|N$, and each integer $c$ such that $c|N|c^2$, the function $\lambda_{[c], p}(n_1)$ is defined as follows, $$ \lambda_{[c], p}(n_1)=   (n_1, p^\infty)^{1/2}\left( \int_{\Z_p^\times} \left| W_p\left(\begin{pmatrix}un_1/c^2&0\\0&1 \end{pmatrix}\begin{pmatrix}0&1\\-1&0\end{pmatrix}
 \begin{pmatrix}1 & 1/c\\0&1 \end{pmatrix} \right)\right|^2 du \right)^{1/2},$$  These local functions $ \lambda_{[c], p}(n_1)$ were studied in detail in~\cite{NPS}, and completely explicit (and remarkably simple) expressions for them were proved. For the purposes of this lemma, we only need the following bound, which follows from~\cite[Prop. 3.12 (b)]{NPS} and~\cite[Cor. 3.13]{NPS}:
$$\lambda_{[N/M], p}(n_1) \ll (n_1, p^\infty)^{1/4} .$$ This gives us $\lambda_{[N/M], N}(n_1) \ll n_1^{1/4} $
 which implies that for each $n_1 |N^\infty$, we have the bound

$$\sum_{\substack{n_0 \bmod{M} \\ (n_0, M) = (n_0, N)=1}} |\lambda_{\sigma, N}(n_0n_1)|^2 \ll Mn_1^{1/2}.$$

In particular, if $Y$ is any positive \emph{integer}, we deduce (by completing the residue classes) that \begin{equation}\label{keylocalfour}\sum_{\substack{1 \le |n_0| \le M Y \\ (n_0, N)=1}} |\lambda_{\sigma, N}(n_0n_1)|^2 \ll YMn_1^{1/2}.\end{equation}

Hence \begin{align*}\sum_{1\le |n| \le X}|\lambda_{\sigma, N}(n)|^2 &=  \sum_{\substack{1 \le n_1 \le X \\ n_1 |N^\infty}} \sum_{\substack{1 \le |n_0| \le \frac{X}{n_1} \\ (n_0, N)=1}} |\lambda_{\sigma, N}(n_0n_1)|^2 \\ &\le \sum_{\substack{1 \le n_1 \le X \\ n_1 |N^\infty}} \sum_{\substack{1 \le |n_0| \le M \left \lceil \frac{X}{Mn_1} \right \rceil \\ (n_0, N)=1}} |\lambda_{\sigma, N}(n_0n_1)|^2 \\ &\ll \sum_{\substack{1 \le n_1 \le X \\ n_1 |N^\infty}} \left\lceil \frac{X}{Mn_1} \right \rceil M n_1^{1/2}, \qquad \text{ using }\eqref{keylocalfour} \\ &\le \sum_{\substack{1 \le n_1 \le X \\ n_1 |N^\infty}} \left( \frac{X}{Mn_1} + 1 \right)  M n_1^{1/2} \\ &\le \sum_{\substack{1 \le n_1 \le X \\ n_1 |N^\infty}}  (X + M\sqrt{X})   \\ & \ll_{\epsilon} (NX)^\epsilon (X + M\sqrt{X}),
\end{align*}
as required. In the last step above, we have used the fact that there are $\ll_\epsilon (NX)^\epsilon$ integers $n_1$ satisfying $1 \le n_1 \le X, \ n_1 |N^\infty$.
\end{proof}

\begin{remark}It is possible that the error term $M\sqrt{X}$ might be sharpened with more delicate analysis. However, doing so will not lead to any improvement in our main Theorem, and so we do not attempt to do it.

\end{remark}

\begin{lemma}\label{l:secondbound}Suppose $y \ge 1/N$. We have the bound $$\sum_{1\le n \le X} \left|\lambda_f\left( \frac{|n| }{(|n|, N^\infty) } \right)K_{ir}(2 \pi |n| y) \right|^2 \ll_{R,\epsilon} X^{1-2\Im(r)} y^{-2\Im(r)} (NX)^\epsilon.$$

\end{lemma}
\begin{proof}Recall that $\Im(r) \in [0, 1/2]$. Using the well-known bound $K_{ir}(u) \ll u^{-\Im(r) - \epsilon}$ for $u>0$, we see that it suffices to prove that \begin{equation}\label{reqe}\sum_{1\le n \le X} \left|\lambda_f\left( \frac{|n| }{(|n|, N^\infty) } \right)\right|^2|n|^{-2\Im(r)} \ll_{R,\epsilon} X^{1-2\Im(r)} (NX)^\epsilon.\end{equation}

The left side of~\eqref{reqe} equals \begin{align*} \sum_{\substack{1 \le n_1 \le X \\ n_1 |N^\infty}} \sum_{\substack{1 \le n_0 \le \frac{X}{n_1} \\ (n_0, N)=1}} \lambda_f(|n_0|)^2 (n_0n_1)^{-2\Im(r)} & \le \sum_{\substack{1 \le n_1 \le X \\ n_1 |N^\infty}} \sum_{1 \le n_0 \le \frac{X}{n_1} } \lambda_f(|n_0|)^2(n_0n_1)^{-2\Im(r)} \\ & \ll_\epsilon \sum_{\substack{1 \le n_1 \le X \\ n_1 |N^\infty}} X^{1 - 2 \Im(r)} (NRX)^\epsilon \\ & \ll_\epsilon  X^{1 - 2 \Im(r)} (NRX)^\epsilon
\end{align*}

where we have used the bound \begin{equation}\sum_{1\le n \le X} \left|\lambda_f(n)\right|^2|n|^{-2\Im(r)} \ll_\epsilon X^{1-2\Im(r)} (NRX)^\epsilon\end{equation} which follows from the analytic properties of the Rankin-Selberg $L$-function (e.g., combine equation (2.28) of~\cite{harcos-michel} with the usual partial summation).

\end{proof}

We can now prove Proposition~\ref{prop:fourier}.
Recall the Fourier expansion $$h(z) =y^{1/2} \sum_{n \ne 0} \rho_\sigma(n)K_{ir}(2 \pi |n| y) e(nx).$$ The tail of the sum, with $|n|y > N^\epsilon$, is negligible because of the decay of the Bessel function. For the remaining terms, we apply the Cauchy-Schwarz inequality:

\begin{align*}|h(z)|^2 &\ll_{R, \epsilon} N^\epsilon y |\rho(1)|^2 \left(\sum_{1\le |n| \le \frac{N^\epsilon}{y}} |\lambda_{\sigma, N}(n)|^2 \right)\left(\sum_{1\le |n| \le \frac{N^\epsilon}{y}}\left|\lambda_f\left( \frac{|n| }{(|n|, N^\infty) } \right)K_{ir}(2 \pi |n| y) \right|^2  \right) \\& \ll_{R, \epsilon} N^\epsilon  \left(\frac{1}{Ny} + \frac{M}{N\sqrt{y}} \right).
\end{align*}

where we used Lemma~\ref{l:firstbound}, Lemma~\ref{l:secondbound} and the estimate $|\rho(1)|^2 \ll N^{\epsilon-1}$ due to Hoffstein-Lockhart~\cite{HL94}.

\section{Proof of the main result}

Recall that $f$ is a Hecke-Maass cuspidal newform for the group $\Gamma_0(N)$ with Laplace eigenvalue $\lambda = 1/4 + r^2$ such that $\langle f, f \rangle_{\Gamma_0(N)} =1$ and $|r| \le R$.

We first deal with the question of proving $$\|f\|_\infty \ll_{R,\epsilon} N^{-1/12 + \epsilon}.$$ Let $z \in \H$. We need to prove that $|f(z)| \ll_{R,\epsilon } N^{-1/12 + \epsilon}.$ Let $M, W, \sigma$ be as in Proposition~\ref{prop:gap} and put $x'+iy' = z' := \sigma^{-1}Wz$. Put $g :=f | \sigma$.
Then, as $f|W = \pm f$, it follows that $|g(z')| = |f(z)|$. So it suffices to prove that $|g(z')| \ll_{R,\epsilon} N^{-1/12 + \epsilon}.$

We first consider the case $M \ge N^{1/12}$. In this case we have by Proposition~\ref{prop:gap} that $y' \gg M^2/N \ge N^{-5/6}$. Using Proposition~\ref{prop:fourier}, we conclude that $$g(z') \ll_{R, \epsilon} N^\epsilon\max\left( \frac1{(N\cdot N^{-5/6})^{1/2}}, \frac{M^{1/2}}{N^{1/2}(M^2/N)^{1/4}}\right) \ll_{R, \epsilon} N^{-1/12+\epsilon}.$$

So we may henceforth assume that $M<N^{1/12}$. Furthermore, we may henceforth assume that $y' < N^{-5/6}$, for otherwise, Proposition~\ref{prop:fourier} finishes the job again. For future reference, we record this as follows.

\begin{equation}\label{e:yconstraints}1 \le  M \ll N^{1/12}, \qquad  \frac{M^2}{N} \ll y'  \ll N^{-5/6}. \end{equation}

Put $\Gamma = \Gamma_0(N;M)$. We note that $g$ is a Maass cusp form on $\Gamma$ that satisfies $$M^{1-\epsilon} \ll_\epsilon \frac{\phi(M)}{(M,2)} = \langle g, g \rangle_{\Gamma} \le M.$$  Let $g' = \frac{g}{\langle g, g \rangle_{\Gamma}^{1/2} }$. Then $\langle g', g' \rangle_{\Gamma} = 1$. It suffices to show that \begin{equation}\label{reeq}|g'(z') |^2\ll_{R,\epsilon} M^{-1}N^{-1/6 + \epsilon}.\end{equation}

By Corollary~\ref{heckecor}, $g'$ satisfies for all $l \equiv 1 \bmod{M}$, $$\sum_{\gamma \in \Gamma_0(N;M)  \bs \Delta(l,N;M) }g'|\gamma = \lambda_f(l) l^{1/2} g.$$ Define $$\mathcal{P}:= \{p \text{ prime }| p \equiv 1 \bmod{M}, \ \Lambda<p<2\Lambda\}, \quad \mathcal{P}^2:= \{p^2: p\in \mathcal{P}\},$$ and $$x_l:=\begin{cases} \sgn(\lambda_f(l)), & l\in \mathcal{P} \cup \mathcal{P}^2 \\ 0& \text{ otherwise}. \end{cases}$$ By embedding the cusp form $g'$ into an orthonormal basis of Maass cusp forms for $\Gamma$ and then using the amplifier method as in~\cite{harcos-templier-1} (with the amplifier $x_l$ defined above), we obtain the inequality $$\frac{\Lambda^2}{M^2} |g'(z')|^2 \ll_{R,\epsilon} (N\Lambda)^\epsilon \sum_{l \ge 1}\frac{y_l}{\sqrt{l}} N(z,l,N;M),$$ where $y_l$ is defined as in Prop.~\ref{prop:ampl}. Using~\eqref{amplesum} and~\eqref{reeq}, we conclude that it suffices to prove the following inequality for some $\Lambda \ge M^2$,

$$\frac{M^2}{\Lambda} + y'N_0+ \frac{M\Lambda^{1/2}}{\sqrt{N}} + \frac{M^2\Lambda^2}{N} \ll N^{-1/6}.$$

Choosing $\Lambda = N^{1/3}$,  using~\eqref{e:yconstraints}, and using $N_0 \le N^{1/2}$, the required inequality follows.
\bigskip

Next, we suppose that there is no integer $M'$ in the range $1<M'<N^{1/6}$ such that $M'^2$ divides $N$. We need to prove that  $$\|f\|_\infty \ll_{R,\epsilon} N^{ \epsilon} \max(N^{-1/6}, N^{-1/4}N_0^{1/4}).$$ Let $z \in \H$. We need to prove that $|f(z)| \ll_{R,\epsilon }N^{ \epsilon} \max(N^{-1/6}, N^{-1/4}N_0^{1/4}).$ Let $M, W, \sigma$ be as in Proposition~\ref{prop:gap} and put $x'+iy' = z' := \sigma^{-1}Wz$. Put $g :=f | \sigma$.
Then, as $f|W = \pm f$, it follows that $|g(z')| = |f(z)|$. So it suffices to prove that $|g(z')| \ll_{R,\epsilon} N^{ \epsilon} \max(N^{-1/6}, N^{-1/4}N_0^{1/4}).$

As before, we can reduce to the case $M<N^{1/6}$ using Prop.~\ref{prop:fourier}. \emph{By our assumption, it follows that $M=1$}. Furthermore, we may assume that $y' \le \frac1{\sqrt{NN_0}}$ for otherwise, Proposition~\ref{prop:fourier} finishes the job again. Proceeding exactly as before, the amplification method reduces us to having to prove the following inequality,

$$\frac{1}{\Lambda} + y'N_0+ \frac{\Lambda^{1/2}}{\sqrt{N}} + \frac{\Lambda^2}{N} \ll_{\epsilon} N^{ \epsilon} \max(N^{-1/3}, N^{-1/2}N_0^{1/2}).$$

Choosing $\Lambda = N^{1/3}$ and using $y' \le \frac1{\sqrt{NN_0}}$, the inequality follows. The proof is complete.

\bibliography{refs-que}

\end{document}